\documentclass[a4paper,12pt]{amsart}
\pdfoutput=1

\usepackage[top=2cm,bottom=2cm,left=3cm,right=2.5cm]{geometry}
\usepackage[english]{babel}
\usepackage{bbold}
\usepackage{stmaryrd}
\usepackage{verbatim}
\usepackage[utf8]{inputenc}
\usepackage{amsmath}
\numberwithin{equation}{subsection}
\usepackage{amssymb}
\usepackage{amsthm}
\usepackage{graphicx}
\usepackage{mathtools}
\usepackage[T1]{fontenc}
\usepackage{amsthm}
\usepackage{setspace}
\usepackage{amsfonts}
\usepackage{eucal}
\usepackage{mathrsfs}
\usepackage{enumitem}

\usepackage{pictexwd,dcpic}
\usepackage{tikz}
\usepackage{tqft}
\usepackage{pigpen}
\usepackage{xcolor}
\usepackage{pdfpages}
\usepackage{geometry}

\usepackage{hyperref}
\hypersetup{
	colorlinks=true,
	linkcolor=blue,
	filecolor=magenta,      
	urlcolor=cyan,
}

\usepackage{thm-restate}
\usepackage{tikz-cd}
\usepackage{float}

\newcounter{tmp}

\usepackage{blkarray}

\usepackage[all,cmtip]{xy}
\UseComputerModernTips

\newtheorem*{theorem*}{Theorem}
\newtheorem{theorem}{Theorem}[section] 
\newtheorem{lemma}[theorem]{Lemma}
\newtheorem*{lemma*}{Lemma}      

\newtheorem{proposition}[theorem]{Proposition}

\newtheorem*{conj*}{Conjecture}
\newtheorem{conj}{Conjecture}[theorem]
\theoremstyle{definition}
\newtheorem{definition}[theorem]{Definition}
\newtheorem*{definition*}{Definition}

\theoremstyle{remark}
\newtheorem{remark}[theorem]{Remark}
\numberwithin{equation}{section}

\theoremstyle{definition}

\newcommand{\C}{{\mathbb{C}}}
\newcommand{\N}{{\mathbb{N}}}
\newcommand{\PP}{{\mathbb{CP}}}

\newcommand{\M}{{\mathcal{M}}}

\newcommand{\0}{{0}}

\newcommand{\x}{\boldsymbol{x}}
\newcommand{\z}{\boldsymbol{z}}

\DeclareMathOperator{\Span}{Span}

\DeclareMathOperator{\Reg}{Reg}
\DeclareMathOperator{\Sing}{Sing}

\DeclareMathOperator{\ssp}{Spec}

\newcommand{\D}{\mathbb{D}}

\title[]
{Periodic points of weakly post-critically finite all the way down maps} 

\author[ Van Tu Le ]{Van Tu Le}
\address{Van Tu Le\\ Dipartimento di Matematica\\
	Università degli Studi di Roma "Tor Vergata"\\Via della Ricerca Scientifica 1 - 00133 Roma\\ Italy} 
\email{vantule@axp.mat.uniroma2.it}
\date{\today}

\begin{document}
\begin{abstract}
	We study eigenvalues along periodic cycles of post-critically finite endomorphisms of $\PP^n$ in higher dimension. It is a classical result when $n = 1$ that those values are either $0$ or of modulus strictly bigger than $1$. It has been conjectured in [Van Tu Le. Periodic points of post-critically algebraic holomorphic endomorphisms. \textit{Ergodic Theory and Dynamical Systems}, pages 1–33, 2020] that the same result holds for every $n \ge 2$. In this article, we verify the conjecture for the class of weakly post-critically finite all the way down maps which was introduced in [Matthieu Astorg. Dynamics of post-critically finite maps in higher dimension. \textit{Ergodic Theory and Dynamical Systems}, 40(2):289–308, 2020]. This class contains a well-known class of post-critically finite maps constructed in [Sarah Koch. Teichmüller theory and critically finite endomorphisms. \textit{Advances in Mathematics}, 248:573–617, 2013.]. As a consequence, we verify the conjecture for Koch maps.
\end{abstract}

	
\maketitle
\section{Introduction}
Let $f \colon X \to X$ be a holomorphic endomorphism of a connected complex manifold $X$. The critical set $C(f,X)$ consists of points $x$ where the derivative $D_x f$ is not invertible. The post-critical set $PC(f,X)$ is defined as
\[
PC(f,X) = \bigcup\limits_{j \ge 1} f^{\circ j}(C(f,X)).
\]
We shall write $C(f),PC(f)$ instead of $C(f,X), PC(f,X)$ when $X$ is clearly stated. The map $f$ is called \textit{post-critically finite}, or PCF for short, if $PC(f,X)$ has finitely many irreducible components. A point $z \in X$ is a \textit{periodic point of period $m$} if $f^{\circ m}(z) = z$ and $m$ is the smallest positive integer satisfying such a property. An eigenvalue of $f$ along a periodic point $z$ of period $m$ (or along the cycle of $z$) is an eigenvalue of the derivative $D_z f^{\circ m}$.  

When $X = \PP^1$, it is a classical result that the eigenvalue along a periodic cycle of a PCF rational map is either $0$ or of modulus strictly bigger than $1$. The result relies on the local study in a neighborhood of a periodic cycle, which is far from being fully understood in higher dimension. The author proposed the following conjecture generalizing this fact.

\begingroup
\setcounter{tmp}{\value{conj}}
\setcounter{conj}{0} 
\renewcommand\theconj{\Alph{conj}}

\begin{conj}\label{con 1}
	Let $f \colon \PP^n \to \PP^n, n \ge 2$ be a post-critically finite endomorphism of degree $d \ge 2$ and $\lambda$ be an eigenvalue of $f$ along a periodic cycle. Then either $\lambda = 0$ or $|\lambda| >1$.
\end{conj}
Conjecture \ref{con 1} has been proved when $ n = 2$ and when the cycle is outside the post-critical set in any dimension. The main challenge lies in the case when the cycle is inside the post-critical set, especially when it belongs to a singular invariant irreducible component. However, even when all the periodic components in the post-critical set are smooth, the conjecture remains open. In \cite{astorg2018dynamics}, Astorg introduced a class of so-called \textit{weakly post-critically finite all the way down maps} (see Definition \ref{def : weakly all the way down}). Roughly speaking, this class of maps consists of PCF maps whose post-critical sets have only smooth periodic components and the restrictions of each maps on those periodic components are also weakly PCF all the way down. Using an additional condition, Astorg indirectly verified the conjecture for at least one eigenvalue (see \cite[Theorem B]{astorg2018dynamics}). In this article, we prove the conjecture for the class of weakly PCF all the way down without any condition. 
\begin{theorem}\label{ thm : CPn}
	Let $f \colon \PP^n \to \PP^n, n \ge 2$ be a post-critically finite endomorphism of degree $d \ge 2$ and $\lambda$ be an eigenvalue of $f$ along a periodic cycle. If $f$ is weakly post-critically finite all the way down, then either $\lambda = 0$ or $|\lambda | >1$. 
\end{theorem}
As a consequence, we are able to verify the conjecture for the class of PCF maps constructed in \cite{koch2013teichmuller}. The special case of eigenvalues at fixed points of Koch maps induced by post-critically finite unicritical polynomials was studied in \cite{le2022fixed}. In \cite{valentin_prep}, Huguin independently studies the dynamics of Koch maps. In particular, he shows explicitly that they are PCF all the way down and obtains more precise information about their eigenvalues along periodic cycles, and thus, also verifies the conjecture for Koch maps.


\textbf{Acknowledgement:} The author would like to thank Valentin Huguin for his comments on an early version of this article and for communicating Lemma \ref{lem : semiconjugacy}. The author would also like to thank his advisors, Xavier Buff and Jasmin Raissy, for introducing him the subject. The author was supported by MIUR Excellence Department Project awarded to the Department of Mathematics, University of Rome Tor Vergata, CUP E83C18000100006.
\section{Eigenvalues of holomorphic germs with smooth invariant hypersurfaces}
We start by proving two local results which will be helpful later and might be of independent interest. Its concern the eigenvalues of a holomorphic germ at $0$ under the presence of an invariant hypersurface which is smooth.

\begin{proposition}\label{prop : preperiodic germs}
	Let $f: (\C^n,\0) \rightarrow (\C^n, \0), \gamma : (\C,0) \to (\C^n,0)$ be non-constant holomorphic germs and $\Gamma$ be a germ of a smooth hypersurface at $0$. Assume that $\Gamma$ is invariant under $f$, i.e. ${f(\Gamma) \subset \Gamma}$, and $\gamma((\C,0)) \not\subset \Gamma, f \circ \gamma  \subset \Gamma$. Then \[\ssp(D_{\0}f) = \ssp(D_{\0} f|_{T_{\0} \Gamma}) \cup \{0\}.\]
\end{proposition}
\begin{proof}
We can choose a local coordinate $\x = (x_1,\ldots,x_n)$ such that $\Gamma = \{ x_n =0\}$. Since $f(\Gamma) \subset \Gamma$, the germ $f$ has an expansion of the form\[
f(\x) = \left(\begin{matrix}
A & *\\
0 & \lambda
\end{matrix}\right)\x + (g_1(\x),\ldots,g_{n-1}(\x), x_n h(\x))
\]
where $g_i(\x) = O(\|\x\|^2), h(\x) = O(\|\x\|)$ and $A \in \M(n-1,\C)$ such that $\ssp(A) = \ssp(D_0 f|_{T_0 \Gamma})$. We need to show that $\lambda = 0$. 

In the coordinate $\x$, the germ $\gamma \colon (\C,0) \to (\C^n,0)$ is of the form $
\gamma = (\gamma_1 ,\ldots,\gamma_n).$
Since $f \circ \gamma \subset \Gamma$, we have  
\[
\lambda \gamma_n + \gamma_n h\circ \gamma \equiv 0.
\]
The hypothesis $\gamma((\C,0)) \not\subset \Gamma$ implies that  $\gamma_n \not\equiv 0$. Hence, \[\lambda + h \circ \gamma \equiv 0.\] Note that $h,\gamma$ are non-constant and the expansion at $0$ of $f$ and $\gamma$ have no constant terms. Therefore, $\lambda = 0$.


\end{proof}
\begin{proposition}\label{prop : two smooth tangential invariant hypersurfaces} 
	Let $f \colon (\C^{n+1},0) \to (\C^{n+1},0)$ be a proper holomorphic germ and $\Gamma_1,\Gamma_2$ be two distinct germs at $0$ of smooth hypersurfaces which are invariant under $f$ such that $T_0 \Gamma_1 = T_0 \Gamma_2 $. Denote by $\lambda$ the eigenvalue of the linear endomorphism of $T_{0} \C^{n+1} /{T_{0} \Gamma_1}$ induced by $D_{0} f$.  Then there exist an integer $d \ge 1$ and $\mu_1,\ldots,\mu_d \in \ssp(D_{0}f|_{T_{0} \Gamma_1})$ such that 
	\[
	\lambda = \prod_{i =1}^{d} \mu_d.
	\]
\end{proposition}
Note that $\mu_1,\ldots,\mu_d$ in the statement of Proposition \ref{prop : two smooth tangential invariant hypersurfaces} are not necessarily distinct. 
\begin{proof}
	We can choose a local coordinate $(\x,x_{n+1}) = (x_1,\ldots,x_n,x_{n+1})$ such that $\Gamma_1 = \{x_{n+1} = 0\}$. Since $T_0 \Gamma_2 = T_0 \Gamma_1 = \{x_{n+1} = 0 \}$ and $\Gamma_2 \neq \Gamma_1$, by Implicit function theorem (see \cite[Proposition 1.1.11]{huybrechts2006complex}), there exists a holomorphic germ $\varphi : (\C^n,0) \to (\C,0)$ such that
	\begin{equation}\label{eq : equation two hypersurfaces invariant}
	(\x,x_{n+1}) \in \Gamma_2 \Leftrightarrow  x_{n+1} = \varphi(\x).
	\end{equation} 
Moreover, $\varphi$ has an expansion of the form
\[
	\varphi(\x) = P_d(\x) + O(\|\x\|^{d+1}) 
	\]
	where $P_{d}(\x)$ is a homogeneous polynomial of degree $d$ for some $d \geq 2$.
	
	Since $f (\Gamma_1) \subset \Gamma_1$, $f$ has an expansion of the form
	\[
	f(\x,x_{n+1}) =\left( \begin{matrix}
	A & v \\
	0 & \lambda
	\end{matrix} \right) \left(\begin{matrix}
	\x \\x_{n+1}
	\end{matrix} \right)  + \left( \begin{matrix}
	(g_1,\ldots,g_{n})(\x,x_{n+1})\\
	x_{n+1} h(\x,x_{n+1})
	\end{matrix} \right)
	\]
	where $v \in \C^n, g_i (\x,x_{n+1}) = O(\|(\x,x_{n+1})\|^2), h (\x,x_{n+1}) = O(\|(\x,x_{n+1})\|)$ and $A \in \M(n,\C)$ satisfies that $\ssp(A) = \ssp(D_{0}f|_{T_{0}\Gamma_1})$. Our objective is to relate $\lambda$ and $\ssp(A)$. 
	
Let $(\z,\varphi(\z))$ be a point in $\Gamma_2$. Then $f(\z,\varphi(\z))$ is of the form
\[
f(\z,\varphi(\z)) =\left( \begin{matrix}
A & v \\
0 & \lambda
\end{matrix} \right) \left(\begin{matrix}
\z \\\varphi(\z)
\end{matrix} \right)  + \left( \begin{matrix}
(g_1,\ldots,g_{n})(\z, \varphi(\z))\\
\varphi(z) h(\z,\varphi(\z))
\end{matrix} \right) 
\]
Since $f(\Gamma_2) \subset \Gamma_2$, by \eqref{eq : equation two hypersurfaces invariant}, we have 
\begin{equation}\label{eq : f(Gamma2) in Gamma2}
\lambda \varphi(\z) + \varphi(z) h(\z,\varphi(\z)) = \varphi\big(A \z + v \varphi(\z) + (g_1,\ldots,g_n)(\z,\varphi(\z)) \big).
\end{equation}
Recall that $\varphi(\z) = P_d(\z) + O(\|\z\|^{d+1})$. By comparing the homogeneous terms of degree $d$ in the expansion at $0$ of both sides of \eqref{eq : f(Gamma2) in Gamma2}, we deduce that 
\[
\lambda P_d = P_d \circ A.
\]
Then the conclusion of Proposition \ref{prop : two smooth tangential invariant hypersurfaces} is a direct consequence of the following linear algebra lemma which is due to Huguin.
\begin{lemma}\label{lem : semiconjugacy} 
	Let $A \colon \C^n \to \C^n, n \ge 1$ be a linear endomorphism, $P_d \colon \C^n \to \C$ be a non-vanishing homogeneous polynomial of degree $d \ge 1$ and $\lambda \in \C$ be such that 
	\[
	\lambda P_d = P_d \circ A.
	\] 
	Then there exist $\mu_1,\ldots,\mu_d \in \ssp(A)$ such that 
	\[
	\lambda = \prod_{i =1}^{d} \mu_d.
	\]
\end{lemma}
\begin{proof}[Proof of Lemma \ref{lem : semiconjugacy} ]
There exists a multilinear symmetric function $\Phi \colon (\C^n)^d \to \C$ such that for all $\x \in \C^n, P_d(\x) = \Phi(\x,\ldots,\x)$. We can use the following equivalent expression (see \cite{thomas2014polarization})
\begin{equation}\label{eq : polar expression}
\Phi(\x_1,\ldots,\x_d) = \frac{1}{d!} \sum\limits_{J \subset \{1,\ldots,d\}} (-1)^{d - |J|} P_d(\sum\limits_{j \in J}\x_j)
\end{equation}
and  $\lambda P_d = P_d \circ A$ to deduce that
\[
\Phi(A(\x_1),\ldots,A(\x_d)) = \lambda \Phi(\x_1,\ldots,\x_d).
\]
Let $(e_1,\ldots,e_n)$ be a basis of $\C^n$ such that the matrix of $A$ in this basis is upper triangular. For $j \in \{0,\ldots,n\},$ set $F_0 =\{0\}$ and $F_j = \Span(e_1,\ldots,e_j)$. Then for every $j \in \{1,\ldots,n\}$, there exist $\mu_j \in \ssp(A), \epsilon_j \in F_{j-1}$ such that
\[
A(e_j) = \mu_j e_j + \epsilon_j.
\]
We equip $\{1,\ldots,n\}^d$ with the lexicographic order. Since $P_d$ is non-vanishing, so is $\Phi$. Hence, there exists the minimum multi-index $(j_1,\ldots,j_d) \in \{1,\ldots,n\}^d$ such that $\Phi(e_{j_1},\ldots,e_{j_d}) \neq 0$. In particular, for all $(k_1,\ldots,k_d) \in \{1,\ldots,n\}^d$ 
\begin{equation}\label{eq : lexico order}
(k_1,\ldots,k_d) < (j_1,\ldots,j_d) \Rightarrow \Phi(F_{k_1} \times \ldots \times F_{k_d}) = 0.
\end{equation}
We have
\[\begin{array}{rcl}
\lambda \Phi(e_{j_1},\ldots,e_{j_d}) & = & \Phi(A(e_{j_1}),\ldots,A(e_{j_d})\\
&=& \Phi( \mu_{j_1}e_{j_1} + \epsilon_{j_1},\ldots,\mu_{j_d}e_{j_d} + \epsilon_{j_d} ).\\
\end{array}\]
Note that $\epsilon_{j_d} \in F_{j_d - 1}$. By expanding the right hand side of the last equality and using \eqref{eq : lexico order}, we have
\[
\lambda \Phi(e_{j_1},\ldots,e_{j_d}) = \prod_{k = 1}^d \mu_{j_k} \Phi(e_{j_1},\ldots,e_{j_d}).
\]
Since $\Phi(e_{j_1},\ldots,e_{j_d}) \neq 0$, we have $\lambda = \prod_{k = 1}^d \mu_{j_k}$.
\end{proof}
We can see the conclusion of Proposition \ref{prop : two smooth tangential invariant hypersurfaces} follows directly from Lemma \ref{lem : semiconjugacy}.
\end{proof}
\section{Periodic points of weakly post-critically finite all the way down}
We recall the notion of weakly PCF all the way down which was introduced in \cite{astorg2018dynamics}.
\begin{definition}
	Let $f : X \to X$ be a PCF endomorphism of connected complex manifold. Set $f_0 \coloneqq f, \Omega_0 \coloneqq X, P_0 \coloneqq PC(f,X)$. The following notions are defined inductively on $m \ge 0$ : if the restriction of $f_m$ on irreducible components of $\Omega_m$ is either unbranched covering or PCF, then the following are true.
	\begin{itemize}
		\item $\Omega_{m+1}$ is the union of $f_m$-periodic irreducible components of $P_m$ and $k_m$ is the least common multiple of the periods.
		\item $f_{m+1}$ is the restriction of $f_m^{k_m}$ on $\Omega_{m+1}$.
		\item If every irreducible component of $P_{m}$ is smooth, $P_{m+1}$ is the union of the post-critical sets of $f_{m+1}$ restricted to each irreducible component of $\Omega_{m+1}$. 
	\end{itemize}
\end{definition}
\begin{definition}\label{def : weakly all the way down}
	Using the same notations as the definition above.
	\begin{itemize}
		\item If for some $k \in \N$, the restriction of $f_m$ on each irreducible component of $\Omega_m$ is PCF for every $m \le k$, then we say $f$ is \textit{PCF of order $k+1$.}
		\item If $f$ is PCF of order $\dim X$, we say that $f$ is \textit{PCF all the way down.}
		\item If for all $k \le \dim X - 1$, the restriction of $f_m$ to each irreducible component of $\Omega_m$ is either unbranched or PCF, then we say that $f$ is \textit{weakly PCF all the way down. }
	\end{itemize}
\end{definition}
\begin{remark}\label{rm : remarks on definition of all the way down } Let $f \colon X \to X$ be a PCF endomorphism of a connected complex manifold.
	\begin{enumerate}[label = (\arabic*)]
		\item\label{itm : smooth of post-critical}The function $f$ is then PCF of order $1$ and vice versa. However, if $f$ is PCF of order $2$, then necessarily the periodic irreducible components of $PC(f)$ are smooth hypersurfaces. Note that it does not exclude the existence of singular irreducible components of $PC(f)$ which are preperiodic to a cycle of smooth periodic components. For example, the following PCF map of $\PP^2$, which is due to Rong in \cite{rong2008fatou}, $[z:w:t] \mapsto [z^d - w^{d-1}t : -w^d : -t^d], d \ge 2$ is not PCF of order $2$ since it has a periodic critical curve of period $2$ which is singular.
		\item \label{item : restriction is also all the way down }It follows from the definition that if $f$ is (weakly) PCF all the way down, then for every $k < \dim X - 1$, the restriction of $f_m$ to each irreducible component of $\Omega_m$ is also (weakly) PCF all the way down. 
	\end{enumerate}
\end{remark}
It is therefore natural to work with the restriction of an endomorphism $f \colon X \to X$ to an irreducible invariant analytic set. In our setting, we are particularly interested in the restriction of an endomorphism of $\PP^n$ on an invariant irreducible algebraic set. We employ the notion of \textit{polarized endomorphisms} of projective algebraic sets. Polarized endomorphisms are originally defined in more general settings, see \cite{Fakhruddin}. However, the extra technicalities are not needed for our proof. Hence, we use the following definition.

\begin{definition}
	Let $X$ be a smooth algebraic set of $\PP^n$. An endomorphism $f \colon X \to X$ is called a \textit{polarized endomorphism} of $X$ if there exists an endomorphism $F \colon \PP^n \to \PP^n$ of degree at least $2$ such that $F (X) = X$ and $F|_X = f$.
\end{definition}
We need to extend the main results of \cite{van2020periodic} to the case of polarized endomorphisms of smooth projective algebraic sets.
\begin{theorem}\label{thm : polarized}
	Let $f \colon X \to X$ be a polarized endomorphism of a smooth projective algebraic set $X$ and $\lambda$ be an eigenvalue of $f$ along a periodic point $z \in X$. Assume that $f$ is weakly post-critically finite, i.e. $f$ is either unbranched or post-critically finite. 
	\begin{enumerate}[label = \alph*.]
		\item \label{itm : outside} If $z \notin PC(f)$, then $|\lambda|>1$.
		\item \label{itm : tangential}If $z \in \Reg PC(f)$ and $\lambda$ is an eigenvalue of \[D_{z} f \colon T_{z} X/_{T_{z} PC(f)} \to T_{z} X/_{T_{z} PC(f)}\] then either $\lambda = 0$ or $|\lambda| >1$.
	\end{enumerate}
\end{theorem}
The proof follows by almost the same proof of \cite[Theorem 1.3, Proposition 4.1]{van2020periodic}, except for a mild modification in the first case \ref{itm : outside} when $f$ is unbranched. For the sake of completeness, we include here the strategy and the readers are invited to \cite{van2020periodic} for more details.
\begin{proof}[Sketch of the proof] By lifting to homogeneous coordinates in $\C^n$, we can assume that: $X$ is an algebraic cone in $\C^n$ for some $n \ge 2$ such that $X \setminus \{0\}$ is smooth and $z \in X \setminus \{0\}$, there exists a regular homogeneous polynomial endomorphism $F \colon \C^n \to \C^n$, i.e. $F^{-1}(0) = \{0\}$, such that $F(X) = X, F|_{X} = f$. In particular, $f$ inherits several important properties of $F$ as a regular polynomial endomorphism, see \cite[Proposition 3.2]{van2020periodic}.
	
Assume that $ z \notin PC(f)$. Note that if $f$ is unbranched then $PC(f) = \emptyset$. We proceed by the following steps.

\noindent \textbf{Step 1.} Set $X^* = X \setminus PC(f)$ and let $ \pi \colon \tilde{X} \to X^*$ be the universal covering. Since $f \colon 
f^{-1}(X^*) \to X^*$ is a covering, by homotopy lifting property, there exists a map $g \colon \tilde{X} \to \tilde{ X}$ such that 
\[
f \circ \pi \circ g = \pi
\]
and $g$ fixes a point $[z]$ such that $\pi([z]) = z$. The point $[z]$ represents the homotopy class of constant paths in $X^*$ at $z$ and locally $g$ is the continuation along paths of the unique inverse branch of $f$ which fixes $z$. In particular, $g$ is holomorphic and $\lambda^{-1}$ is an eigenvalue of $D_{[z]} g$. 

\noindent \textbf{Step 2.} The family $\{ g^{\circ j} \colon \tilde{X} \to \tilde{X} \}$ is normal. This can be proved thanks to the following observation. Let $U \subset X^*$ be an open simply connected neighborhood of $z$ and $h_j \colon U \to X^*, j \geq 1$ be the local inverse branch of $f^{\circ j}$ fixing $z$, then every limit map $h = \lim\limits_{k \to \infty} h_{j_k}$ takes values in $X^*$. Indeed, since $X$ is closed, $h(U) \subset X$. If $f$ is unbranched then $X^* = X$ and hence $h(U) \subset X^*$. If $f$ is post-critically finite then $h(U) \subset X^*$ by Hurwitz's theorem, see \cite[Lemma 3.5]{van2020periodic}. The first consequence of this fact is that $|\lambda| \geq 1$. We shall proceed by contradiction. Suppose that $|\lambda| = 1$.

\noindent \textbf{Step 3.} The normality of $\{g^{\circ j}, j \geq 1\}$ and the assumption $|\lambda| = 1$ imply the existence of a center manifold $M$ in $\tilde{X}$ which is semi-linearizable. More precisely, 
\begin{itemize}
	\item $M \neq \emptyset$ is a closed complex submanifold of $\tilde{X}$.
	\item $g|_{M} \colon M \to M$ is an automorphism.
	\item $[z] \in M$ and $T_{[z]} M$ is the sum of eigenspaces associated to eigenvalues of $D_{z} f$ which are of modulus $1$.
	\item $\pi(M)$ is bounded in $\C^n$.
	\item There exists a holomorphic map $\Phi \colon M \to T_{[z]}M$ such that $\Phi([z]) = 0, D_{[z]} \Phi = Id$ and 
	\[
	D_{[z]} g \circ \Phi = \Phi \circ g.
	\]
\end{itemize}
In particular, $\lambda$ is not a root of unity.

\noindent \textbf{Step 4.} Let $v \in T_{[z]} M$ be an associated eigenvector of $\lambda$ and $\C v$ the complex line of direction $v$ in $T_{[z]} M$. Denote by $\Gamma$ the irreducible component of $\Phi^{-1}(\C v)$ containing $[z]$. Then we can prove that $\Gamma$ is smooth and the map 
\[
\Phi|_{\Gamma} \colon \Gamma \to \Phi(\Gamma) = \D(0,R)
\]
is a biholomorphism with $R \in (0, +\infty)$. Let us sum up the construction by the following diagram.
\[\begin{tikzcd}        
\D(0,R) \arrow[rrrr, bend left=20, "\tau  = \pi \circ (\Phi|_{\Gamma})^{-1}"] \arrow[d, "\lambda^{-1}v"'] &\Gamma \subset  T_{[z]} M  \arrow[l, "\Phi"] \arrow[d, "D_{[z]}g"']& M \arrow[l, "\Phi"] \arrow[hook, r, ] \arrow[d,"g"'] & \tilde{X} \arrow[r, "\pi"] \arrow[d, "g"'] & X^*\\
\D(0,R) &\Gamma   \subset T_{[z]} M \arrow[l, "\Phi"] & M \arrow[l, "\Phi"] \arrow[hook, r, ] & \tilde{X} \arrow[r, "\pi"] & f^{-1}(X^*) \subset X^* \arrow[u, "f"]\\
\end{tikzcd}\]

\noindent \textbf{Step 5. } Set $\tau = \pi \circ (\Phi|_{\Gamma})^{-1} \colon \D(0,R) \to \pi(M)$. Since $\pi(M)$ is bounded in $\C^n$, the radial limit 
\[
\tau_\theta = \lim\limits_{r \to R^-} \tau(r e^{i \theta})
\]
exists for almost every $\theta \in [0, 2\pi)$. It follows from \cite[Proposition 3.17]{van2020periodic} that if $\tau_\theta$ exists then $\tau_\theta \in PC(f) = X \setminus X^*$. We have two cases.
\begin{itemize}
	\item If $f$ is unbranched then $PC(f) = \emptyset$, i.e. $\tau_\theta$ cannot exist. This yields a contradiction.
	\item If $PC(f) \neq \emptyset$ then $PC(f)$ is defined by a holomorphic map $Q \colon X \to \C$. The map $Q \circ \tau$ has vanishing radial limit almost everywhere on the boundary, hence $Q \circ \tau \equiv 0$. In particular, $Q \circ \tau(0) =0$, which implies that $z \in PC(f)$. This is a contradiction to the assumption that $z \notin PC(f)$. 
\end{itemize}
Hence the proof is complete for the case $z \notin PC(f)$. The second case follows the same lines using the $(X^*,z)$-- homotopy introduced in \cite[$\S$4]{van2020periodic}.
\end{proof}
We are ready to the main theorem about eigenvalues of polarized endomorphisms which are weakly PCF all the way down.
\begin{theorem}
	Let $f \colon X \to X$ be a polarized endomorphism of a smooth projective algebraic set $X$ and $\lambda$ be an eigenvalue of $f$ along a periodic cycle. If $f$ is weakly post-critically finite all the way down, then either $\lambda = 0$ or $|\lambda|>1$. 
\end{theorem}
\begin{proof} We proceed by induction on the dimension of $X$. If $\dim X = 1$, by Riemann -- Hurwitz formula, $X$ has genus zero or one. In both cases, the theorem is classical (see \cite{milnor2011dynamics} or \cite[$\S$ 5]{van2020periodic}). Assume that the theorem is true for all polarized weakly post-critically finite all the way down endomorphisms of smooth projective algebraic sets of dimension at most $k - 1$. Replacing $f$ by a high enough iterate, we can assume that $\lambda$ is an eigenvalue of $f$ at a fixed point $z \in X$ and every periodic irreducible component of $PC(f)$ is invariant. We want to prove that 
\[
\ssp D_{z} f \subset \overline{\D}^c \cup \{0\}
\]
where $\overline{\D}^c = \{ \lambda \in \C \mid |\lambda|>1 \}.$

Thanks to Theorem \ref{thm : polarized}, we can assume that $z \in PC(f)$ and let $\Gamma$ be an invariant irreducible component of $PC(f)$ containing $z$. By Remark \ref{rm : remarks on definition of all the way down } \ref{itm : smooth of post-critical}, $\Gamma$ is a smooth hypersurface in $X$. By induction hypothesis and Remark \ref{rm : remarks on definition of all the way down } \ref{item : restriction is also all the way down }, we have 

\begin{equation}\label{eq : spec tangential}
\ssp D_{z}f|_{T_{z} \Gamma} \subset \overline{\D}^c \cup \{0\}.
\end{equation}
If $z \in \Reg PC(f)$, then the result is proved by using Theorem \ref{thm : polarized} \ref{itm : tangential} and \eqref{eq : spec tangential}. If $z \in \Sing PC(f)$, then $z$ belongs to the intersection of $\Gamma$ with another irreducible component $\Sigma$ of $PC(f)$. Note that $\Sigma$ is not necessarily in $\Omega_1$. In particular, $\Sigma$ can be singular at $z$. By passing to some iterate, we have three cases.
\begin{enumerate}[label = \arabic*.]
	\item\label{itm : case transverse} $\Sigma \subset \Omega_1$; $\Sigma$ and $\Gamma$ intersect transversally, i.e. $T_{z} \Gamma \neq T_{z} \Sigma$. Then, obviously $\ssp(D_{z} f) = \ssp(D_{z}f|_{T_z \Gamma}) \cup \ssp(D_{z}f|_{T_z \Sigma})$, hence the results follows from induction hypothesis.
	\item\label{itm : case tangent} $\Sigma \subset \Omega_1$; $\Sigma$ and $\Gamma$ intersect tangentially, i.e. $T_{z} \Gamma = T_{z} \Sigma$. We can conclude by applying Proposition \ref{prop : two smooth tangential invariant hypersurfaces} and \eqref{eq : spec tangential}
	\item\label{itm : case preperiodic} $f(\Sigma) = \Gamma$. We wish to apply Proposition \ref{prop : preperiodic germs}. More precisely, we need a holomorphic $\gamma \colon (\C,0) \to (\Sigma,z)$ such that $\gamma((\C,0)) \not\subset \Gamma$. In order to obtain such a map, we choose an open neighborhood of $z$ and a local coordinate $(x_1,\ldots,x_{\dim X})$ centered at $z$ and $\Gamma \cap U = \{x_1 = 0\}$. Let $E$ be a hyperplane passing through $z$ so that the orthogonal projection $\sigma \colon \Sigma \to E$ is proper (see \cite[Lemma 1, $\S$ 3.4]{chirka2012complex}). Then $\sigma(\Sigma \cap \Gamma)$ is an analytic set in $E$. In particular, $\sigma(\Sigma \cap \Gamma)$ is nowhere dense in $E$. Hence, we can find a complex line $\mathfrak{l}$ such that $\mathfrak{l} \not\subset \sigma(\Sigma \cap \Gamma)$. Let $\Delta$ be an irreducible component of $\sigma^{-1}(\mathfrak{l})$ containing $z$. Then $\Delta \not\subset \Gamma$. Let $\gamma \colon (\D,0) \to (\Delta,z)$ be the Puiseau parameterization. Then $\gamma((\C,0)) \not\subset \Gamma$ and $f\circ \gamma ((\C,0)) \subset (f(\Sigma),z) = (\Gamma,z)$. Hence, we can apply Proposition \ref{prop : preperiodic germs} and conclude the proof by using \eqref{eq : spec tangential}.
\end{enumerate}
\end{proof}
Since every endomorphism of $\PP^n$ is polarized by definition, Theorem \ref{ thm : CPn} follows immediately from Theorem \ref{thm : polarized}. Finally, let us consider the conjecture for a PCF map $f$ constructed by Sarah Koch in \cite{koch2013teichmuller}. A remarkable property of $f$ is that $PC(f)$ is a union of projective hyperplanes. Consequently, $f$ is PCF all the way down (see \cite[Corollary 1]{astorg2018dynamics}). Thus, we can apply Theorem \ref{ thm : CPn} and the conjecture is proved for the map $f$. 
%

\bibliography{ref}
\bibliographystyle{alpha}
\end{document}